\documentclass[12pt]{article}
\usepackage{bm}
\usepackage{epsf}
\usepackage{amssymb}
\usepackage{graphics}
\usepackage{amsthm}
\usepackage{amsfonts}
\usepackage{latexsym}
\usepackage{mathrsfs}
\usepackage{graphicx}
\usepackage{amsxtra}
\usepackage{mathtools}
\usepackage{mathrsfs}
\usepackage{anyfontsize}
\usepackage{t1enc}
\usepackage{xfrac}
\usepackage[inline]{asymptote}

\usepackage{doi}
\providecommand{\eprint}[1]{}
\renewcommand{\eprint}[1]{arXiv:\href{http://arxiv.org/abs/#1}{#1}}

\usepackage{titlesec}
\titleformat{\chapter}[display]
{\normalfont%
    \normalfont
    \bfseries}{\chaptertitlename\ \thechapter}{20pt}{%
    \Large 
    }

\usepackage[textwidth=165mm,textheight=235mm]{geometry}

\usepackage{tcolorbox}

\newcommand{\beqn}{\begin{eqnarray}}
\newcommand{\eeqn}{\end{eqnarray}}

\newcommand{\be}{\begin{equation}}
\newcommand{\ee}{\end{equation}}

\newcommand{\ve}{\varepsilon}

\newtheorem{theorem}{Theorem}
\newtheorem{lemma}{Lemma}

\makeatletter
\DeclareFontFamily{OMX}{MnSymbolE}{}
\DeclareSymbolFont{MnLargeSymbols}{OMX}{MnSymbolE}{m}{n}
\SetSymbolFont{MnLargeSymbols}{bold}{OMX}{MnSymbolE}{b}{n}
\DeclareFontShape{OMX}{MnSymbolE}{m}{n}{
    <-6>  MnSymbolE5
   <6-7>  MnSymbolE6
   <7-8>  MnSymbolE7
   <8-9>  MnSymbolE8
   <9-10> MnSymbolE9
  <10-12> MnSymbolE10
  <12->   MnSymbolE12
}{}
\DeclareFontShape{OMX}{MnSymbolE}{b}{n}{
    <-6>  MnSymbolE-Bold5
   <6-7>  MnSymbolE-Bold6
   <7-8>  MnSymbolE-Bold7
   <8-9>  MnSymbolE-Bold8
   <9-10> MnSymbolE-Bold9
  <10-12> MnSymbolE-Bold10
  <12->   MnSymbolE-Bold12
}{}

\let\llangle\@undefined
\let\rrangle\@undefined
\DeclareMathDelimiter{\llangle}{\mathopen}%
                     {MnLargeSymbols}{'164}{MnLargeSymbols}{'164}
\DeclareMathDelimiter{\rrangle}{\mathclose}%
                     {MnLargeSymbols}{'171}{MnLargeSymbols}{'171}
\makeatother

\title{How doth the random triangle}

\author{Theodore D. Drivas and Michael  Retakh}

\date{}

\begin{document}

{\large

\maketitle

\vspace{-4mm}

Charles L. Dodgson, also known as Lewis Carroll, in his book \textit{Pillow problems} from 1893, asked for the likelihood of a random triangle to be obtuse (having an angle exceeding $90^\circ$) \cite[Question No. 58]{C58}:
\begin{figure}[htb]\centering
    \includegraphics[width=.8\columnwidth]{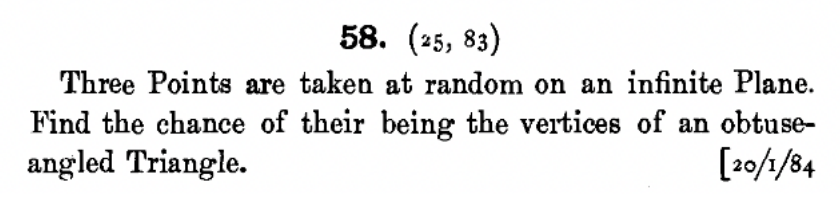} 
\end{figure}

True to form, Dodgson's question and his own proposed solution are whimsical (see \cite{G93,P94}).  For one, there are many ways to randomly draw a triangle but -- characteristically -- the way Dodgson defined it with a uniform  distribution on the plane is nonsense. This aside, he arrived at the intriguing conclusion that obtuse triangles are substantially more likely than acute.  

Our aim here is to study a generalization of Dodgson's question, and to examine more closely this apparent bias towards obtuse triangles. 
A random triangle on the plane (or - more generally - in $d$-dimensional Euclidean space) will be defined as three points drawn independently from a prescribed probability distribution.  We require that the probability distribution give zero mass to degenerate configurations of points on the plane. For example, the distribution could be uniform on some compact subset of the space (excluding points and straight line segments).   We ask:
\vspace{2mm}

\noindent  \textbf{Question:} \textit{Across all probability distributions on the plane, how unlikely can random obtuse triangles be? How likely?}
\vspace{2mm}

This question is reminiscent of Sylvester's Four-Point Problem \cite{S65,P89,BS11} which asks the probability of four "random" points to be in a convex position; the answer depends strongly on the distribution, but there are some limitations.

As for our question, it is easy to see that \emph{there exist probability distributions on the plane for which obtuse triangles appear with probability one}.  Indeed, take a distribution that uniformly weights a circular arc no longer than a semicircle.  Any three points on such an arc define an obtuse triangle, and so with full probability they occur.  Such a distribution is admittedly singular but among distributions that are absolutely continuous with respect to Lebesgue, the probability can be made as close to one as desired by approximation.

Less obvious is how unlikely obtuse triangles may be. One might be tempted to think that there should be some symmetry between acute and obtuse triangles, but this is not so (as Dodgson's original conclusion suggested).  In fact, it is easy to see that: 
\emph{Given any probability distribution in the plane $\mathbb{R}^2$, the probability of a random triangle being obtuse is  at least $\sfrac{1}{4}$}. 

To see this, let us replace the old scheme to generate a random triangle "select three points in the plane according to the given probability measure" with a new completely equivalent scheme (by independence): "select four points in the plane, according to the given probability measure, and then select one of those points uniformly at random to be discarded." The quadrilateral formed by the four points must have at least one obtuse angle, so one of the four possible triangles must be obtuse. Therefore, the probability of the triangle being obtuse after discarding one point at random is bounded below by $\sfrac{1}{4}$.

In view of this argument, it is natural to try to improve the 25\% bound upward by considering  $n$–tuples.
Using this idea, we prove:
\vspace{2mm}

   \begin{tcolorbox}
\begin{theorem}\label{thm1}
The probability of an obtuse triangle on $\mathbb{R}^2$ is at least $\sfrac{1}{3}$. 
\end{theorem}
\end{tcolorbox}

We will later comment on the issue of sharpness of this lower-bound, as well as interesting behavior for triangles embedded in higher dimensional space.

\begin{proof}[Proof of Theorem \ref{thm1}]
We  build off the  $25\%$ argument above. Taking 5 points, we have 5 groups of $4$ points, guaranteeing at least $5$ obtuse triangles. However, we count each triangle twice, so we are guaranteed only  $\lceil \frac{5}{2} \rceil = 3$ obtuse triangles, giving a lower bound of $\frac{3}{\binom{5}{2}} = 0.3$.

In order to exploit $n$–tuples, we require the following

\begin{lemma}\label{2dlem}
Let $t_n$ be the minimum number of obtuse triangles in a configuration of $n$ planar points. Then, $t_n = \frac{\binom{n}{3} - \lfloor \frac{n}{3} \rfloor}{3}$.
\end{lemma}

\begin{proof}[Proof of Lemma \ref{2dlem}]
We first argue for the following recurrence relation:
\be\label{recursion}
t_{n+1} = \lceil t_n \cdot \tfrac{n+1}{n-2} \rceil.
\ee
Taking $n+1$ points gives us $n+1$ groups of $n$ points, or a minimum of $t_n(n+1)$ obtuse triangles. However, multiplying gives us a total of $\binom{n}{3} \cdot (n+1)$ triangles instead of $\binom{n+1}{3}$ triangles. To fix this overcounting, we multiply to get $$t_{n+1} = \left\lceil t_n(n+1) \cdot \tfrac{\binom{n+1}{3}}{\binom{n}{3} \cdot (n+1)} \right\rceil = \left\lceil t_n(n+1) \cdot \tfrac{1}{n-2} \right\rceil = \left\lceil t_n \cdot \tfrac{n+1}{n-2} \right\rceil.$$

Having now established the recurrence, using the initial value of $t_4 = 1$ (by the quadrilateral argument), we now have a recursive definition for the minimum number of obtuse triangles. We aim to establish the claimed formula.

We proceed with induction. We can verify our base case of $t_4 = 1$, as $\frac{\binom{4}{3} - \lfloor \frac{4}{3} \rfloor}{3} = \frac{4-1}{3} = 1$. 
For our inductive step, we must prove that $$\tfrac{\binom{n+1}{3} - \lfloor \frac{n+1}{3} \rfloor}{3} = \left\lceil \tfrac{\binom{n}{3} - \lfloor \frac{n}{3} \rfloor}{3} \cdot \tfrac{n+1}{n-2} \right\rceil = \left\lceil \tfrac{\binom{n+1}{3} - \lfloor \frac{n}{3} \rfloor \cdot \frac{n+1}{n-2}}{3} \right\rceil.$$

First, we note that $\binom{n}{3} \equiv \left\lfloor \frac{n}{3} \right\rfloor \pmod 3$ for all $n$, as can be seen through casework based on $n \mod 9$. Therefore, the left side of our equation above is indeed an integer.

We are thus able to remove the ceiling to create the inequality $\left\lfloor \frac{n+1}{3} \right\rfloor \le \left\lfloor \frac{n}{3} \right\rfloor \cdot \frac{n+1}{n-2} < \left\lfloor \frac{n+1}{3} \right\rfloor + 3.$
A little casework is required, so we let $n = 3k, 3k+1,$ or $3k+2$ for any integer $k$, and we proceed with simplifying the inequality.

If $n=3k$, our inequality becomes $k \le k \cdot \frac{3k+1}{3k-2} < k+3.$ The left half is evidently true, and the right half can be simplified into $k(3k+1) < (k+3)(3k+2)$, which is true for all $k > 1$. As the smallest possible value of $n$ is $4$, and because all steps are reversible, our inductive step holds if $n \equiv 0 \pmod 3$.

Similarly, if $n = 3k+1$, our inequality becomes $k \le  \frac{k(3k+2)}{3k-1} < k+3.$ The left half is once again evidently true, and the right half holds for all $k > \frac{1}{2}$, or all $n \ge 3$. Therefore, our inductive step holds if $n \equiv 1 \pmod 3$.

Finally, if $n = 3k+2$, our inequality becomes $k+1 \le  \frac{k(3k+3)}{3k} < k+4.$ Since the middle term simplifies to $k+1$, the inequality clearly holds for all $k$. Therefore, our inductive step holds if $n \equiv 2 \pmod 3$.

Having examined all  cases, we can conclude that our inductive step holds true for all $n$. By induction, our Lemma is proved.
\end{proof}

We are now in the position to prove Theorem 1. Since the total number of triangles that can be created using $n$ points is $\binom{n}{3}$ and at least $\frac{\binom{n}{3} - \lfloor \frac{n}{3} \rfloor}{3}$ must be obtuse, a minimum bound for the probability of a randomly chosen triangle being obtuse, given any distribution, is $\lim_{n\to\infty} \sfrac{\frac{\binom{n}{3} - \lfloor \frac{n}{3} \rfloor}{3}}{\binom{n}{3}} = \frac{1}{3}.$
\end{proof}
\vspace{2mm}

In view of this lower bound, it is natural to ask about its sharpness.  Namely, do there exist probability distributions on the plane for which obtuse triangles occur with likelihood $\sfrac{1}{3}$?    The best we know is: \emph{there exist probability distributions such that the probability of an obtuse triangle is arbitrarily close to $\sfrac{4}{9}$.} The idea is this: let $\mathsf{A}$, $\mathsf{B}$, $\mathsf{C}$ be the vertices of an acute triangle where the angle at $\mathsf{A}$ is very close to right. Let $\delta, \varepsilon$ be small numbers. The measure consists of:
\begin{enumerate}
\item Mass $\frac{1}{3}$ uniformly distributed on a circular arc of length $\delta$ centered at $\mathsf{A}$,
\item Mass $\frac{1}{3}$ uniformly distributed on a circular arc of length $\ve\delta$ centered at $\mathsf{C}$,
\item Mass $\frac{1}{3}$ uniformly distributed on a circular arc of length $\ve^2\delta$ centered at $\mathsf{B}$.
\end{enumerate}
See Figure \ref{figtri}.
\begin{figure}[htb]\centering
    \includegraphics[width=.45\columnwidth]{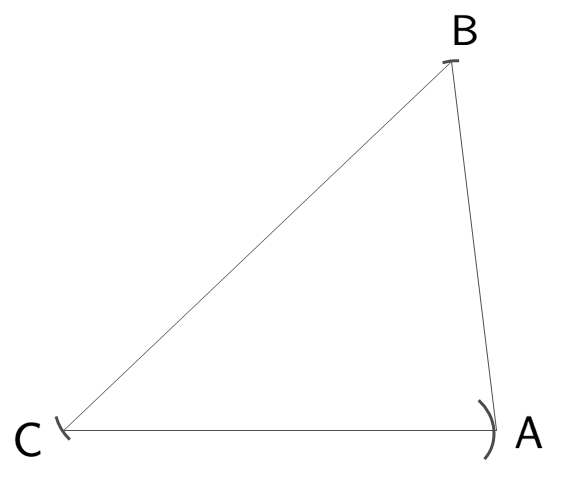} 
    \caption{Planar distribution with $\sfrac{4}{9}$ likelihood of obtuse triangle}
    \label{figtri}
\end{figure}
Assuming the angle at $\mathsf{B}$ is sufficiently close to $\pi/2$ and that $\delta$ is sufficiently small,
one can see that any triangle whose vertices fall at $\mathsf{ABC}$ will be acute, as well as all but a fraction $O(\ve)$ of triangles whose vertices fall at $\mathsf{BBA}$, $\mathsf{AAC}$, or $\mathsf{CCB}$.
This gives a probability of $\sfrac{5}{9} -O(\ve)$ of acute triangles.

The sharpness of the bound $\sfrac{1}{3}$ is open. In view of this, we ask
\vspace{2mm}

   \begin{tcolorbox}
\noindent  \textbf{Open Problem:} \textit{Find the minimal probability $p_*$ of an obtuse triangle.}
\end{tcolorbox}

Theorem 1, together with the above construction, show that $p_*$ lies somewhere in the range $\sfrac{1}{3}\leq p_*\leq \sfrac{4}{9}$ for planar triangles.
\vspace{2mm}

We now discuss a further generalization, where the triangle is generated by independent draws from a probability distribution on $d$--dimensional Euclidean space $\mathbb{R}^d$.  In the case of three-dimensions, we have
\vspace{2mm}

   \begin{tcolorbox}
\begin{theorem}\label{thm2}
The probability of an obtuse triangle on $\mathbb{R}^3$ is at least $\sfrac{1}{11}$. 
\end{theorem}
\end{tcolorbox}

\begin{proof}[Proof of Theorem \ref{thm2}]
The crucial fact is that, for any 6 points in $\mathbb{R}^3$, there will be at least one triple of points that form an obtuse triangle (see e.g. \cite{C61}). Therefore, the extension of the n\"{a}ive  "four point argument" in dimension two to three dimensions would yield $5\%$, or $\sfrac{1}{20}$, since out of every $\binom{6}{3}=20$ triangles, at least one will be obtuse. However, we can raise this bound by applying the same recursive  logic as in two dimensions. Importantly, the key recursion relation \eqref{recursion} from Lemma 1 applies in arbitrary dimensions. Now, we claim that 
\begin{lemma}\label{3dlem}
Let $t_n$ be the minimum number of obtuse triangles in a configuration of $n$ points in $\mathbb{R}^3$. Then, $t_n =  \frac{\binom{n}{3}-2n+k}{11}$ where $k$ depends on the remainder of $n$ when divided by $11$:
\begin{center}
\begin{tabular}{c|c|c|c|c|c|c|c|c|c|c|c}
 $n$\ {\rm mod} {11} & 0 & 1 & 2 & 3 & 4 & 5 & 6 & 7 & 8 & 9 & 10 \\ \hline
 $k$ & 0 & 2 & 4 & 5 & 4 & 0 & 3 & 1 & 4 & 0 & -1
\end{tabular}.
\end{center} 
\end{lemma}
\begin{proof}[Proof of Lemma \ref{3dlem}]
Unfortunately, the proof of the Lemma is quite tedious, as the induction must proceed through eleven different cases. Here we sketch the beginning of this proof; the entire process is similar.

Let us take those cases for which $n \equiv 6 \pmod{11}$. The base case is simple: for every $6$ points, there must be $1$ obtuse triangle, and our formula returns $\frac{20-12+3}{11} = 1$. Let us then assume that our formula is correct when $n \equiv 6 \pmod{11}$. Then, we must prove that $$\tfrac{\binom{n+1}{3}-2(n+1)-10}{11} < \tfrac{\binom{n}{3}-2n+3}{11} \cdot \tfrac{n+1}{n-2} \le \tfrac{\binom{n+1}{3}-2(n+1)+1}{11}.$$ (Note that we substitute the values of $k$ in based on the modulus we are working with, and that we change an equation with a ceiling function into an inequality.)

We can then multiply all sides of the inequality by $11$, and use that $\binom{n}{3} \cdot \frac{n+1}{n-2} = \binom{n+1}{3}$. Therefore, subtracting $\binom{n+1}{3}$ from all sides gives us $$-2n-12 < (-2n+3)\left(\tfrac{n+1}{n-2}\right) \le -2n-1.$$
The inequality is easily shown to be satisfied for all necessary $n$ (e.g. $\geq 6$.)

Therefore, if our formula (and values of $k$) is correct for $n \equiv 6 \pmod {11}$, then it is correct for $n \equiv 7 \pmod {11}$. As mentioned above, we could go through each modulus until we get back to $n \equiv 6 \pmod{11}$, at which point our inductive step will truly be complete and, together with our base case, we will have established our formula for the minimum number of obtuse triangles.
\end{proof}

We now complete the proof of Theorem 2. Dividing $t_n$ by $\binom{n}{3}$, the total number of triangles formed by $n$ points, limits to $\sfrac{1}{11}$ as $n$ goes to infinity, since the growth of $\binom{n}{3}$ far outpaces that of $-2n$. 
Thus, given any probability distribution in three dimensions, an obtuse triangle is no  less likely than $\sfrac{1}{11}$.
\end{proof}

Now for a construction with low probability of obtuse triangles. Consider a probability distribution in $\mathbb{R}^3$, supported
on a spherical cap (probability $p$) together with a small patch at the center of the sphere (probability $1-p$). On the cap, which can be as flat as we like, we choose the probability distribution to look like the triangle construction resulting in a $\sfrac{5}{9}$ acute likelihood in two dimensions.  Then, the acute-probability for this distribution is $3(1-p)p^2$
(two points on the cap and one in the patch) plus $(1-p)^3 x$ 
(three points in the patch, where $x$ = acute-probability
of the patch) plus $\frac{5}{9}p^3$ (three points on the cap).  Now the patch consists of a
tiny little spherical cap, structured in the same way, together with a tiny little
patch at its center, so $x$ is given by $3(1-p)p^2 + (1-p)^3 x'+ \frac{5}{9}p^3,$
where $x'$ is the acute-probability for that tiny little
patch.  Continue in this way -- replace each patch in
turn by a spherical cap with a smaller patch at its center. See Figure \ref{figss} for a depiction of the resulting self-similar distribution.
We are led to solving for $x$ in
$$
    x = 3(1-p)p^2 + (1-p)^3 x+ \tfrac{5}{9}p^3.
$$
The choice of $p=\sfrac{(22 - \sqrt{133})}{13} $ maximizes the likelihood of acute triangles, which is $x=\sfrac{(2 \sqrt{133}-17)}{9}\approx 0.673903$.  Thus, obtuse triangles occur with probability about $0.3261$.
\begin{figure}[htb]\centering
    \includegraphics[width=.8\columnwidth]{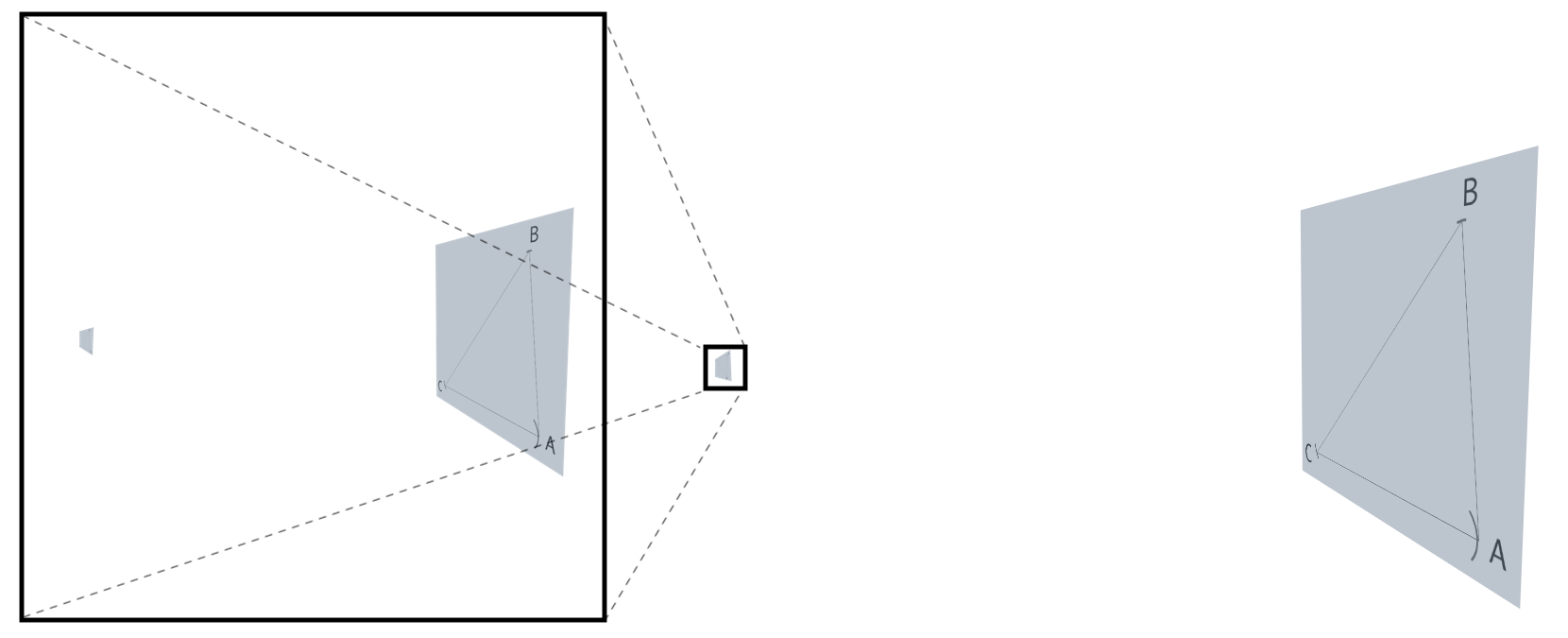} 
    \caption{Self-similar distribution in 3D with low probability of obtuse.}
    \label{figss}
\end{figure}


The above shows that, for triangles in $\mathbb{R}^3$, $\sfrac{1}{11}\leq p_*\leq \sfrac{(26-2\sqrt{133})}{9} \approx  0.3261$.  
\\

Of further interest are bounds for the minimum probability of obtuse triangles for distributions in $\mathbb{R}^d$ as $d \rightarrow \infty$. We begin by extending the recursive argument above to these higher dimensions.

Most importantly, perhaps, is the fact that there is no known sharp bound for the minimum number of points needed to guarantee at least one obtuse triangle  above three dimensions \cite{KZ20}. This can be somewhat circumvented by using the simplest bound, which is $2^d$ points for $d$ dimensions (the maximum size of an acute set in $\mathbb{R}^d$ is known to be between $2^{d-1}-1$ and $2^d-1$). See discussion in \cite{GH19, Z19, KZ20}. 
Therefore,  we can assert that the probability of a random triangle in $\mathbb{R}^d$ being obtuse must be greater than or equal to $\sfrac{1}{\binom{2^d}{3}}$.  However, this bound is clearly not sharp, as we can easily use similar recursion ideas to improve it. Using a computer program for the recursion, we can find the following approximate bounds for several higher dimensions:

\begin{center}
\begin{tabular}{c|c}
Number of Dimensions & (Approximate) Extrapolated Lower Bound \\ \hline
4 & 7.91$\times 10^{-3}$ \\\hline
5 & 1.73$\times 10^{-3}$ \\\hline
6 & 4.07$\times 10^{-4}$ \\\hline
7 & 9.89$\times 10^{-5}$ \\\hline
8 & 2.43$\times 10^{-5}$ \\
\end{tabular}
\end{center} 
Indeed, in dimension 8, compared to the estimate $\sfrac{1}{\binom{2^8}{3}}\approx 3.6\times 10^{-7}$,  the probability reported above is $67$ times larger. Furthermore, these bounds could easily be improved if the upper bound for an acute set is improved, as the program only uses the minimum number of points to guarantee an obtuse triangle.

However, for large numbers of points $n$, the recurrence relation \eqref{recursion} essentially amounts to adding $1$, since multiplying by $\frac{n+1}{n-2}$ only slightly increases $t_n$. In particular, taking the ceiling of the product adds $1$, so the total value of the probability (represented by $\sfrac{t_n}{\binom{n}{3}}$) increases by $\sfrac{1}{\binom{n}{3}}$. Therefore, our recursion, whose closed form was not obvious in low dimensions, tends towards a relatively simple summation for higher dimensions. We have
   \begin{tcolorbox}
\begin{theorem}\label{thm3}
The probability of an obtuse triangle on $\mathbb{R}^d$ is asymptotically at least $\sfrac{3}{2^{2d}}$  as $d \rightarrow \infty$.
\end{theorem}
\end{tcolorbox}
\begin{proof}[Proof of Theorem \ref{thm3}]
We seek to begin this recursion with the minimum number of points that guarantee at least one obtuse triangle, just as we have done for $2$ and $3$ dimensions. For $d >  3$, we simply use the bound $2^d$ points for $d$ dimensions (see discussion above).  Together with the observation that the probability asymptotically increases by $\sfrac{1}{\binom{n}{3}}$, as $d \rightarrow \infty$, the recurrence relation will have a limit closer and closer to $$\sum_{k=2^d}^{\infty} \frac{1}{\binom{k}{3}} = \sum_{k=2^d}^{\infty} \frac{6}{k(k-1)(k-2)} = \frac{3}{2^{2d}-3\cdot2^d+2} \asymp  \frac{3}{2^{2d}}.$$ 
\end{proof}


Given our lower bound for large dimensions, we also seek a construction in high dimensions that would yield a low probability of obtuse triangles. It turns out that, for $d \ge 6$, a uniform distribution on the $(d-1)$--sphere $\mathbb{S}_{d-1}$ yields a probability that is not beaten by any previously-examined distribution, including higher-dimensional analogues of the constructions above or the distributions discussed in \cite{BS96}. (In dimensions $2$, $3$, $4$ and $5$, the self-symmetrical distribution does better than the sphere.)

We now sketch how to compute the probability of a triangle drawn uniformly at random on the surface of the sphere to be obtuse. To start, we can rotate the unit $(d-1)$--sphere to force two of the three points onto the $xy$-plane. Specifically, we rotate the sphere to get one of the points to $(1,0,\dots,0)$, and then rotate about the $x$-axis to land the second point onto the $xy$-plane. Given those points, the triangle is obtuse if and only if the third point is in one of the three spherical caps pictured in Figure \ref{sphpic}.

\begin{figure}
\begin{center}
    \includegraphics[width=.6\columnwidth]{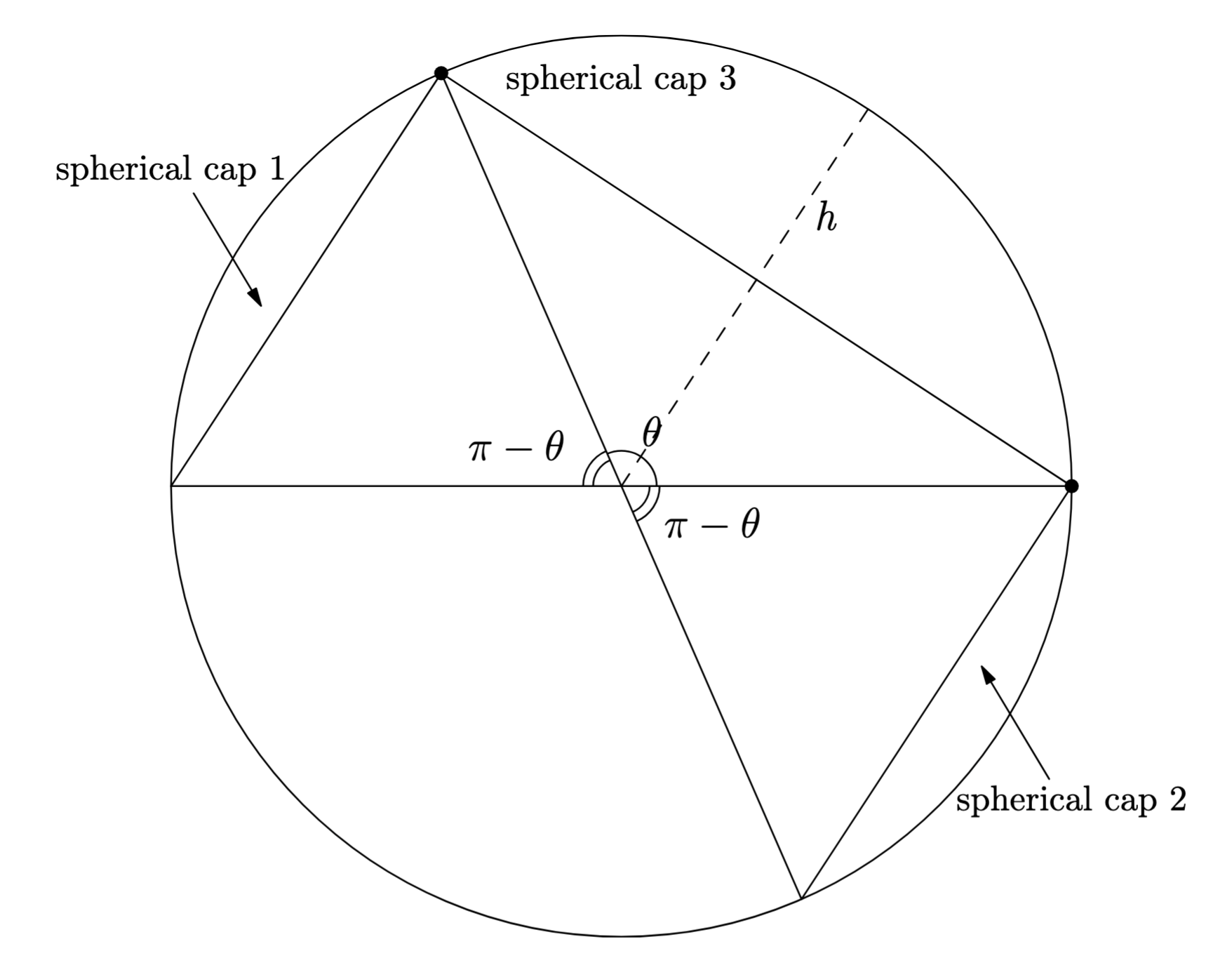} 
    \end{center}
\caption{Cross section of sphere (on $xy$-plane)}
\label{sphpic}
\end{figure}

The area of a hyperspherical cap may be computed in terms of hypergeometric functions (in particular, the regularized incomplete beta function  $I_z(a,b)$). (See \cite{L11}.) Using this, one finds that the probability that a random point is on one of the three spherical caps is 
\be\label{prob}\tfrac{1}{2}I_{\sin^2(\frac{\theta}{2})}\left(\tfrac{d-1}{2},\tfrac{1}{2}\right) + I_{\cos^2(\frac{\theta}{2})}\left(\tfrac{d-1}{2},\tfrac{1}{2}\right). \ee

We must integrate this in the angle $\theta$ from $0$ to $\pi$ to find the total probability of an obtuse triangle, but before we do so, we must adjust our distribution. Because the distribution over the entire sphere is uniform, the distribution of the angle between the two points on the $xy$-plane will not be. When we rotated the sphere about the $x$-axis, we mapped all the points on an $(d-2)$--sphere with radius $\sin\theta$ to a single point. The area of this sphere will be directly proportional to $\sin^{d-2}(\theta)$, so adjusting our distribution gives us that the probability of a uniformly random triangle on $\mathbb{S}_{d-1}$ is obtuse is $$\int_0^\pi \left(\tfrac{1}{2}I_{\sin^2(\frac{\theta}{2})}\left(\tfrac{d-1}{2},\tfrac{1}{2}\right) + I_{\cos^2(\frac{\theta}{2})}\left(\tfrac{d-1}{2},\tfrac{1}{2}\right)\right)\frac{\sin^{d-2}(\theta)\mathrm{d}\theta}{\int_0^{\pi}\sin^{d-2}(\theta)\mathrm{d}\theta}.$$
We remark that in three dimensions (for $\mathbb{S}_2$), this probability is $\sfrac{1}{2}$, making acute and obtuse triangles equally likely (yet another justification about the perfect nature of the 2-sphere).

This precise expression, along with the recursion discussed above, allows us to calculate precise bounds for the minimum probability of a random-spherical  triangle being obtuse in any dimension.  However, we are primarily interested in high dimensions, at which point the measure $\frac{\sin^{d-2}(\theta)\mathrm{d}\theta}{\int_0^{\pi}\sin^{d-2}(\theta)\mathrm{d}\theta}$  approaches $\delta(\theta-\frac{\pi}{2})\mathrm{d}\theta$, a Dirac delta distribution centered at $\frac{\pi}{2}$. Therefore, instead of evaluating a large integral, we can simply let $\theta = \frac{\pi}{2}$ in \eqref{prob} to get an asymptotic approximation of $\frac{3}{2}I_{\frac{1}{2}}\left(\frac{d-1}{2}, \frac{1}{2}\right).$ Thus, combining with Theorem 3, we conclude that the minimal probability of an obtuse triangle on $\mathbb{R}^d$ asymptotically lies in
$$
3\cdot 2^{-2d}< p_*(d) <\tfrac{3}{2}I_{\frac{1}{2}}\left(\tfrac{d-1}{2}, \tfrac{1}{2}\right) \qquad \text{as $d\to\infty$}.
$$
\\

\vspace{-4mm}
\noindent \textbf{Acknowledgements.}  We thank Bob Geroch for many interesting and fun discussions on triangles and their complement. We also thank Daniil Glukhovskiy and Sam Zbarsky for interesting discussions. The work of TDD was partially supported by the NSF CAREER award \#2235395, a Stony Brook University Trustee’s award as well as an Alfred P. Sloan Fellowship.

\qquad\\
Theodore D. Drivas\\
Department of Mathematics\\
Stony Brook University, Stony Brook NY 11790\\
{\tt tdrivas@math.stonybrook.edu}

\qquad\\
Michael  Retakh\\
Ward Melville High School\\
{\tt michael.retakh@gmail.com}

}

\begin{thebibliography}{99}

\bibitem{C58}
C. L. Dodgson, Pillow problems, thought out during wakeful hours. Vol. 2. Macmillan and Company, 1895.
\vspace{-2mm}
\bibitem{C61}
H. T.  Croft. On 6-point configurations in 3-space. Journal of the London Mathematical Society, 1(1), 289-306, (1961).
\vspace{-2mm}
\bibitem{BS96}
B. Eisenberg and R. Sullivan. Random triangles in $n$ dimensions. The American mathematical monthly 103.4 (1996): 308-318.
\vspace{-2mm}
\bibitem{BS11}
B. Eisenberg and R. Sullivan. A modification of Sylvester's four point problem. Mathematics Magazine 84, no. 3 (2011): 173-184.
\vspace{-2mm}
\bibitem{G93}
R. K. Guy. There are three times as many obtuse-angled triangles as there are acute-angled ones. Mathematics Magazine 66.3 (1993): 175-179.
\vspace{-2mm}
\bibitem{KZ20}
A. Kupavskii and D. Zakharov. The right acute angles problem? European Journal of Combinatorics 89 (2020): 103144.
\vspace{-2mm}
\bibitem{GH19}
B. Gerencsér and V. Harangi. Too acute to be true: The story of acute sets. The American Mathematical Monthly 126.10 (2019): 905-914.
\vspace{-2mm}
\bibitem{P94}
S. Portnoy. A Lewis Carroll pillow problem: Probability of an obtuse triangle. Statistical Science (1994): 279-284.
\vspace{-2mm}
\bibitem{P89}
R. E. Pfiefer. The historical development of JJ Sylvester's four point problem. Mathematics Magazine 62.5 (1989): 309-317.
\vspace{-2mm}
\bibitem{Z19}
D. Zakharov. Acute sets. Discrete \& Comp. Geo., 61, 212-217, (2019).
\bibitem{S65}
 J. J. Sylvester.  On a special class of questions on the theory of probabilities. Birmingham British Association Report, 35, 8-9, (1865).
\bibitem{L11}
Shengqiao Li. Concise Formulas for the Area and Volume of a Hyperspherical Cap. Asian Journal of Mathematics and Statistics 4 (1): 66–70, 2011.
\end{thebibliography}
\end{document}